\documentclass[10pt]{amsart}
\usepackage{amssymb}
\usepackage{graphicx}
\usepackage{epsfig}
\usepackage{url}
\usepackage{setspace}
\usepackage{color}
\usepackage{pdflscape}
\theoremstyle{plain}

\newtheorem{thm}{Theorem}[section]
\newtheorem{cor}[thm]{Corollary}
\newtheorem{lem}[thm]{Lemma}
\newtheorem{prop}[thm]{Proposition}
\newtheorem{rem}[thm]{Remark}
\newtheorem{ques}[thm]{Question}
\newtheorem{conj}[thm]{Conjecture}

\def\cal{\mathcal}
\def\bbb{\mathbb}
\def\op{\operatorname}
\renewcommand{\phi}{\varphi}
\newcommand{\R}{\bbb{R}}
\newcommand{\N}{\bbb{N}}
\newcommand{\Z}{\bbb{Z}}
\newcommand{\Q}{\bbb{Q}}

\newcommand{\bs}{\backslash}

\begin{document}
\title[On consecutive 1's in continued fractions expansions]{On consecutive 1's in continued fractions expansions of square roots of prime numbers}
\author{Piotr Miska and Maciej Ulas}

 \subjclass[2010]{}
\thanks{The research of the authors was supported by the Polish National Science Centre grants: UMO-2018/29/N/ST1/00470 (PM) and UMO-2012/07/E/ST1/00185 (MU)}

\begin{abstract}
In this note, we study the problem of existence of sequences of consecutive 1's in the periodic part of the continued fractions expansions of square roots of primes. We prove unconditionally that, for a given $N\gg 1$, there are at least $N\log^{-3/2}N$ prime numbers $p\leq N$ such that the continued fraction expansion of $\sqrt{p}$ contains three consecutive 1's on the beginning of the periodic part. We also present results of our computations related to the considered problem and some related problems, formulate several open questions and conjectures and get some results under the assumption of Hypothesis H of Schinzel.
\end{abstract}

\maketitle

\section{Introduction}\label{section1}
Let $\N$ be the set of non-negative integers, $\N_{+}$ denotes the set of positive integers and $\mathbb{P}$ be the set of prime numbers.

Let $x\in\R\setminus\Q$. It is well known that the number $x$ can be written in the form of infinite continued fraction
$$
x=a_{0}+\cfrac{1}{a_{1}+\cfrac{1}{a_{2}+\cfrac{1}{a_{3}+\cfrac{1}{\ddots}}}}=[a_{0};a_{1},a_{2},\ldots],
$$
where the digits $a_{0},a_{1},\ldots $ can be recursively computed in the following way
$$
\alpha_0=x,\;a_{0}=\lfloor x_{0}\rfloor,\quad \alpha_{k+1}=\frac{1}{\alpha_{k}-a_{k}},\quad a_{k}=\lfloor \alpha_{k}\rfloor.
$$
It is clear that $a_{0}\in\Z$ and $a_{i}\in\N_{+}$ for $i\in\N_{+}$. The number $a_{i}$ is called the $i$-th digit (or $i$-th partial quotient) of the continued fraction for $x$. Moreover, writing $x=[a_{0};a_{1},a_{2},\ldots,a_{k-1},\alpha_{k}]$, then
$$
x=\frac{p_{k-1}\alpha_{k}+p_{k-2}}{q_{k-1}\alpha_{k}+q_{k-2}},
$$
where
\begin{align*}
\begin{array}{lll}
  p_{-1}=1, & p_{0}=p_0(a_0)=a_{0}, & p_{k}=p_k(a_0,a_1,...,a_k)=a_{k}p_{k-1}+p_{k-2}, \\
  q_{-1}=0, & q_{0}=1,     & q_{k}=q_k(a_1,...,a_k)=a_{k}q_{k-1}+q_{k-2}.
\end{array}
\end{align*}
For given $x$, the rational number $p_{k}/q_{k}$ is called the $k$-th convergent of the continued fraction for $x$.

Let us recall that the continued fraction $[a_{0};a_{1},a_{2},\ldots]$ is called periodic if there exists $T\in\N_{+}$ such that for all sufficiently large $n$ we have $a_{n}=a_{n+T}$. The smallest $T$ with this property is called the period of the continued fraction. Then, we write $[a_{0};a_{1},a_{2},\ldots,a_{n-1},\overline{a_{n},\ldots,a_{n-1+T}}]$. One of the classical results in the theory of continued fractions is Lagrange theorem, which says that the irrational number $x$ has periodic continued fraction if and only if $x$ is quadratic irrational. In particular, for any given non-square $D\in\N_{+}$ the continued fraction for the number $\sqrt{D}$ is periodic. Moreover, periodic fraction for $\sqrt{D}$ can be written in the following special form $\sqrt{D}=[[\sqrt{D}];\overline{a_{1},\ldots,a_{n-1},2[\sqrt{D}]}]$ and the sequence $(a_{1},\ldots,a_{n-1})\in\N_{+}^{n}$ is palindromic, i.e., $a_{i}=a_{n-i}$ for $i\in\{1,\ldots,n-1\}$. However, let us note that not all continued fractions of this form represent square roots of integers. Indeed, we have $[1, \overline{1, 3, 1, 2}]=4\sqrt{5}/5$. In the sequel, for $D\in\N_{+}$, we denote by $T_{D}$ the length of the period of continued fraction expansion of $\sqrt{D}$ and by $I_{D,k}=(a_1,\ldots,a_k)$, where $a_1,...,a_k$ are the first $k$ values appearing in the period of continued fraction expansion of $\sqrt{D}$.

The behaviour of digits of continued fraction expansions of irrational numbers is of great interests. Many of basic questions remains unsolved. Probably, the most famous is the following: Let $x$ be a non-quadratic algebraic irrationality. Are the set of digits for the continued fraction expansion for $x$ unbounded? We don't know the answer even in the case of $x=\sqrt[3]{2}$.

On probabilistic ground one can obtain quite precise information about the $i$-th digit in continued fraction expansions of irrational numbers. More precisely, Gauss conjectured that as $i\rightarrow+\infty$ the probability that the $i$-th digit equals $k$ converges to $\log_{2}\left(1+\frac{1}{k(k+2)}\right)$. This conjecture was proved by Kuzmin in 1928 \cite{Kuz}. In particular, the digit which is most probably to occur in continued fraction expansions of irrational numbers is equal to 1. Moreover, further calculations reveals that the probability that the first digit of continued fraction expansion is equal to 1, is $1/2$. This rises an interesting question about behaviour of digits in sequences of algebraic numbers. More precisely, in this paper we are interested in the existence of sequences of consecutive 1's in continued fraction expansions of square roots of prime numbers. As a byproduct of his work on representability of primes by certain quadratic forms, Ska{\l}ba showed, under assumption of certain unproven conjectures, the existence of infinitely many primes $p$ with two consecutive 1's in the continued fraction expansion of $\sqrt{p}$ \cite{Skalba}. However, Vinogradov proved that the set of fractional parts of square roots of primes is dense in the interval $(0,1)$ \cite{Vin}. In fact, an even stronger result is true: the set of fractional parts of square roots of prime numbers is equidistributed in the interval $(0,1)$ (see \cite[Exercise 2, p. 348 and a comment after Theorem 21.3]{IwaKow}). We believe that the problem is still interesting and motivated by Ska{\l}ba observation, we were interested in finding an infinite and easy to describe set of prime numbers, say $A$, with such a property that for each $p\in A$, the first three digits of the periodic part of the continued fraction expansion of $\sqrt{p}$ are equal to 1, i.e., $I_{p,3}=(1,1,1)$. We obtain the required set as the set of prime values of certain nonhomogeneous quadratic polynomial in two variables (and prove that in our set there is no $p$ satisfying $I_{p,4}=(1,1,1,1)$).

In the sequel we will also need the following famous statement from \cite{SchiSier}.

\begin{conj}[Hypothesis H from \cite{SchiSier}]
If $f_{1}(x),\ldots, f_{s}(x)\in\Z[x]$ are irreducible polynomials with positive leading coefficients and such that no integer $n>1$ divides $f_{1}(x),\ldots, f_{s}(x)$ for all integers $x$. Then there are infinitely many $x\in\N_{+}$ such that $f_{1}(x),\ldots, f_{s}(x)$ are simultaneously prime.
\end{conj}

Hypothesis H has many unexpected applications in number theory and arithmetic algebraic geometry. Many of these applications were presented in the classical paper of Schinzel and Sierpi\'{n}ski \cite{SchiSier}. For modern applications see for example \cite{Swi}. As we will see it can be also applied in the realm of continued fractions and allow to get some interesting results.

Let us describe the content of the paper in some details. In Section \ref{section2} we construct (inhomogenous) quadratic polynomial $D(t,d)\in\Z[t,d]$ such that for each prime represented by $D$ for some $t, d\in\N_{+}$ we have $I_{D(t,d),3}=(1,1,1)$ and $I_{D(t,d),4}\neq (1,1,1,1)$. Moreover, invoking one of Iwaniec's results, we show that there are at least $\gg N\log^{-3/2}N$ primes $\leq N$ satisfying these conditions.

In Section \ref{section3} we obtain exact expressions for squares of certain continued fractions and prove several results under Hypothesis H. In particular, we prove that the set $\{L_{1}(m)/T_{p_{m}}:\;m\in\N_{+}\}$, where $L_{1}(m)$ is the number of 1's in the periodic part of the continued fraction expansion of $\sqrt{p_{m}}$, is dense in $[0,1]$.

In Section \ref{section4} we prove some results based on the equidistribution property of the sequence of fractional parts of square roots of prime numbers.

Finally, in last section we present results of our computer experiments. In particular, for each $k\in\{1,\ldots, 20\}$ we present the smallest prime number $p$ such that $I_{p,k}=(1,\ldots,1)$, where in the bracket we have $k$-occurrences of 1. Moreover, for some values of $k$, we give the smallest prime number $p$ such that $T_{p}=k$ and there is no digit 1 in the continued fraction expansion of $\sqrt{p}$. Moreover, based on results of our computations we state some further questions and conjectures.

\section{First results}\label{section2}




Before we state the main result of this section we will give simple characterization of integers $D\in\N_{+}$ satisfying the conditions $T_{D}\leq 8$ and $I_{D,2}=(1,1)$.

\begin{lem}\label{simplelem}
\begin{enumerate}
\item There is no $D\in\N_{+}$ such that $\sqrt{D}=[a;\overline{1,1,2a}]$ for some $a\in\N_{+}$.
\item If $D\in\N_{+}$ and $\sqrt{D}=[a;\overline{1,1,1,2a}]$, then $a=3t-1$ and $D=t(9t-2)$ for some $t\in\N_{+}$.
\item If $\sqrt{D}=[a;\overline{1,1,1,1,2a}]$, then $D\in\N_+$ if and only if $a=5t-2$ for some $t\in\N_{+}$. Then $D=25 t^2-14 t+2$.
\item If $\sqrt{D}=[a;\overline{1,1,x,1,1,2a}]$, then $D\in\N_{+}$ if and only if $x=2u$ and $a=2(2u+1)v-u-1$ for some $u, v\in\N_{+}$. Then $D=(4v-1)((2u+1)^2v-u(u+1))$.
\item If $\sqrt{D}=[a;\overline{1,1,x,x,1,1,2a}]$, then $D\in\N_+$ if and only if $x\in\N_{+}$ with $a=(4x^2+4x+5)t-(2x+1)(x^2+x+1)(x^2+2x+2)$ for some integer $t>(2x+1)(x^2+x+1)(x^2+2x+2)/(4x^2+4x+5)$.
\item If $\sqrt{D}=[a;\overline{1,1,1,x,1,1,1,2a}]$, then $D\in\N_+$ if and only if $x\in\N_{+}$ and $a=3(3x+4)\frac{t}{d(x)}+\frac{1}{2}(x+1)(3x+8)$ for some $t\in\N$, where $d(x)=\gcd(2(6x+7),3(3x+4))$.
\end{enumerate}
\end{lem}
\begin{proof}
The proofs of the statements $(1)-(3)$ are simple consequence of a well known identity
$$
[a;\overline{\underset{k-1\text{ times}}{1,\ldots,1},2a}]^2=\frac{1}{F_{k}}(F_{k}a^2+2F_{k-1}a+F_{k-2}),
$$
where in the bracket we have exactly $k-1$ occurrences of the digit 1.
In the formula above, $F_{n}$ denotes the $n$-th Fibonacci number defined as usual by $F_{0}=0, F_{1}=1$ and $F_{n}=F_{n-1}+F_{n-2}$ for $n\geq 2$.
The proof can be found for example in \cite{Frie} (with a suitable specialisation).

In order to get the fourth equivalence we easily compute that
$$
[a;\overline{1,1,x,1,1,2a}]=\sqrt{\frac{(2a+1)(2(x+1)a+x+2)}{4(x+1)}}.
$$
Writing $(2a+1)(2(x+1)a+x+2)=4(x+1)a(a+1)+2a+x+2$ we see that our continued fraction represents a square root of an integer if and only if $2a+x+2\equiv 0\pmod{4(x+1)}$. In consequence, $x=2u$ for some $u\in\N_{+}$. Then, after simplifications, we are left with the congruence $a+u+1\equiv 0\pmod{2(2u+1)}$ and hence $a=2(2u+1)v-u-1$ for some $v\in\N_{+}$.

We turn our attention to the fifth equivalence. A simple computation reveals that
$$
[a;\overline{1,1,x,x,1,1,2a}]=\sqrt{a^2+\frac{2(2x^2+3x+3)a+x^2+2x+2}{4x^2+4x+5}},
$$
and we need to consider the linear congruence
$$
2(2x^2+3x+3)a+x^2+2x+2\equiv 0\pmod{4x^2+4x+5}.
$$
Here, we treat $a$ as an unknown. From the identity
$$
2(2x^2+3x+3)u-(4x^2+4x+5)v=1,
$$
where $u=(2x+1)(x^2+x+1)$ and $v=2x^3+4x^2+4x+1$ we get that $\gcd(2(2x^2+3x+3), 4x^2+4x+5)=1$. In consequence, all the integer solutions of our congruence are parametrised by the formula $a=(4x^2+4x+5)t-(2x+1)(x^2+x+1)(x^2+2x+2)$, where $x\in\N_{+}$ and $t$ is an integer parameter leading to $a>0$, i.e., $t>(2x+1)(x^2+x+1)(x^2+2x+2)/(4x^2+4x+5)$.

Finally, we concentrate on the last statement. Direct calculation reveals the identity
$$
[a;\overline{1,1,1,x,1,1,1,2a}]=\sqrt{a^2+\frac{2(6x+7)a+4(x+1)}{3(3x+4)}}.
$$
We thus consider the congruence $2(6x+7)a+4(x+1)\equiv 0\pmod{3(3x+4)}$. The necessary and sufficient condition for its solvability is $d(x):=\gcd(2(6x+7),3(3x+4))\mid 4(x+1)$. Because $d(x)=1 (2)$ for $x\equiv 1\pmod{2}$ (for $x\equiv 0\pmod{2}$) the condition is clearly satisfied. Moreover, it is easy to find a particular solution $a=\frac{1}{2}(x+1)(3x+8)$. Thus, the general solution takes the form
$$
a=3(3x+4)\frac{t}{d(x)}+\frac{1}{2}(x+1)(3x+8)
$$
and we get the result.
\end{proof}

We note the following simple consequences of our result.

\begin{cor}
Let $p\in\mathbb{P}$ and suppose that $I_{p,2}=(1,1)$. Then $T_{p}=4$ with $p=7$ or $T_{p}\geq 5$.
\end{cor}

\begin{cor}\label{four1}
It the Hypothesis H is true, there are infinitely many prime numbers $p$, such that $\sqrt{p}=[a;\overline{1,1,1,1,2a}]$ for certain $a\in\N_{+}$.
\end{cor}
\begin{proof}
The polynomial $D(t)=25 t^2-14 t+2$ attains odd values for odd integers $t$, has positive leading coefficient and is irreducible in $\Q[t]$. Moreover, $p\nmid D(0)=2$ for each odd prime number $p$. Thus, from Hypothesis H the polynomial $D$ represents infinitely many prime numbers.
\end{proof}

In the next theorem we construct explicit quadratic polynomial in two variables, such that if $p=D(t,d)$ is a prime for some $t, d\in\N_{+}$, then  $I_{p,3}=(1,1,1)$ and $I_{p,4}\neq (1,1,1,1)$. It is clear that the existence of infinitely many prime numbers $p$ satisfying the property $I_{p,3}=(1,1,1),  I_{p,4}\neq (1,1,1,1)$ follows from Vinogradov result. However, the fact that this set contains all primes represented by specific polynomial in two variables cannot be deduced.

\begin{thm}\label{three1}
Let $t,d\in\N_{+}$ and write $D(t,d)=(4t+3d+5)^2+5t+4d+6$. We have the following properties of the continued fraction expansion of $\sqrt{D(d,t)}$.
\begin{enumerate}
\item We have $I_{D(d,t),3}=(1,1,1)$.
\item The inequality $T_{D(d,t)}\geq 7$ is true. Moreover, $T_{D(d,t)}=7$ if and only if $(d,t)=(1,3)$ with $\sqrt{D(1,3)}=5\sqrt{17}$.
\item We have $T_{D(d,t)}=8$ if and only if
$$
d=(3x-4)u + 3 + x - x^2,\quad t=2(3u-x-2)\;\mbox{with}\; u>\frac{1}{3}(x+2),
$$
where $x, u\in\N_+$. Then $D(d,t)=(9u-3x-1)((3x+4)^2u-(x+1)(3x^2+6x+2))$.
\item We have $I_{D(d,t),4}=(1,1,1,1)$ if and only if $d\in\{1,2\}, t\in\N_{+}$. Then $D(1,t)=(t+2)(16t+37)$ and $D(2,t)=(t+3)(16t+45)$. In particular, in the set $\cal{D}=\{p\in\mathbb{P}:\;p=D(d,t)\;\mbox{for some}\;d,t\in\N_{+}\}$ there is no $p$ satisfying $I_{p,4}=(1,1,1,1)$.
\end{enumerate}
\end{thm}
\begin{proof}
Let us write $\sqrt{D(d,t)}=[a_{0};\overline{a_{1},\ldots, a_{k-1},2a_{0}}]$. Because $(4t+3d+5)^2<D(t,d)<(4t+3d+5+1)^2$ we get
the equality $a_{0}=[\sqrt{D(t,d)}]=4t+3d+5$. Now
$$
\alpha_{1}=\frac{1}{\sqrt{D(d,t)}-(4t+3d+5)}=\frac{\sqrt{D(d,t)}+4t+3d+5}{5t+4d+6}.
$$
It is clear that $\alpha_{1}>1$ and observe that
$$
2-\alpha_{1}=\frac{6t+5d+7-\sqrt{D(d,t)}}{5t+4d+6}>\frac{6t+5d+7-(a_{0}+1)}{5t+4d+6}=\frac{2t+2d+1}{5t+4d+6}>0.
$$
This implies $\alpha_{1}<2$ and in consequence $a_{1}=[\alpha_{1}]=1$.

Next
$$
\alpha_{2}=\frac{1}{\alpha_{1}-1}=\frac{t+d+1+\sqrt{D(d,t)}}{3t+2d+5}
$$
and
$$
2-\alpha_{2}=\frac{5t+3d+9-\sqrt{D(d,t)}}{3t+2d+5}>\frac{5t+3d+9-(a_{0}+1)}{3t+2d+5}=\frac{3+t}{3t+2d+5}>0.
$$
In consequence $1<\alpha_{2}<2$ and $a_{2}=[\alpha_{2}]=1$.

Finally, we observe that
$$
\alpha_{3}=\frac{1}{\alpha_{2}-1}=\frac{2t+d+4+\sqrt{D(d,t)}}{4t+4d+3}
$$
and thus
$$
2-\alpha_{3}=\frac{6t+7d+2+\sqrt{D(d,t)}}{4t+4d+3}>\frac{6t+7d+2-(a_{0}+1)}{4t+4d+3}=\frac{2(t+2d-2)}{4t+4d+3}>0.
$$
In consequence $a_{3}=[\alpha_{3}]=1$ and the first part of our theorem is proved.

In order to prove the second part of our theorem we show that there is no $(d,t)\in\N_{+}\times\N_{+}$ satisfying $T_{D(d,t)}\in\{4,5,6\}$.

If $T_{D(d,t)}=4$ then from identity $[a;\overline{1,1,1,2a}]=\sqrt{3(a+1)(3a+1)}/3$ we get $a=3b-1$ for some $b\in\N_{+}$. Thus we get the equalities
$$
3b-1=4t+3d+5,\quad b(9b-2)=D(d,t)
$$
and in consequence $b=(4t+3d+6)/3$. With this $b$ we observe that $b(9b-2)-D(d,t)=-(t+3)/3$ - a contradiction.

If $T_{D(d,t)}=5$ then from the identity $[a;\overline{1,1,1,1,2a}]=\sqrt{5(5a^2+6a+2)}/5$ we get $a=5b-2$ for some $b\in\N_{+}$. In consequence we have
$$
5b-2=4t+3d+5,\quad 25 b^2-14 b+2=D(d,t)
$$
and thus $b=(4t+3d+7)/5$. However, with $b$ chosen in this way we observe that $25b^2-14b+2-D(d,t)=(t+2d-2)/5$ - a contradiction.

If $T_{D(d,t)}=6$ then from the identity $[a;\overline{1,1,1,1,1,2a}]=\sqrt{2(8a^2+10a+3)}/4$ we get a contradiction with integrality of $a$.

Tying all the obtained observations together we get the required inequality $T_{D(d,t)}\geq 7$.

If $T_{D(d,t)}=7$ then necessarily $\sqrt{D(d,t)}=[a;\overline{1,1,1,1,1,1,2a}]$ with $a=4t+3d+5$ and some $d, t\in\N_{+}$. The identity  $[a;\overline{1,1,1,1,1,1,2a}]=\sqrt{13(13 a^2+16 a+5)}/13$ implies $4t+3d+5=13b-6$ for some $b\in\N_{+}$. In consequence we need to solve the system of the Diophantine equations
$$
13b-6=4t+3d+5,\quad 169 b^2-140 b+29=D(d,t).
$$
We have $b=(4t+3d+11)/13$ and with such $b$ we obtain the equation
$$
169 b^2-140 b+29-D(d,t)=(t+4d-7)/13=0.
$$
The only solution in positive integers of the equation $t+4d-7=0$ is $(d,t)=(1,3)$ and we get the result.

Next, if $T_{D(d,t)}=8$ then necessarily $\sqrt{D(d,t)}=[a;\overline{1,1,1,x,1,1,1,2a}]$ with $a=4t+3d+5$ and some $d, t\in\N_{+}$. The identity
$$
[a;\overline{1,1,1,x,1,1,1,2}]=\sqrt{\frac{(9d+12t+17)(3(3x+4)d+4(3x+4)t+17x+22)}{3(3x+4)}}.
$$
implies that we need to solve the Diophantine equation.
$$
(9d+12t+17)(3(3x+4)d+4(3x+4)t+17x+22)=3(3x+4)D(d,t)
$$
or equivalently
$$
\frac{6d-(3x-4)t-2(5x+1)}{3(3x+4)}=0.
$$
The above equation has a parametric solution
$$
d=(3x-4)u + 3 + x - x^2,\quad t=2(3u-x-2),\quad u\in\Z.
$$
With this choice of $d, t$ we get $D(d,t)=(9u-3x-1)((3x+4)^2u-(x+1)(3x^2+6x+2))$ and observe the equivalence
$$
d,t\in\N_{+}\wedge D(d,t)>0 \quad \Longleftrightarrow u>\frac{1}{3}(x+2).
$$
\bigskip

Finally, in order to prove the last statement we assume that in the continued fraction expansion of $\sqrt{D(d,t)}$ we have $a_{4}=1$. Based on our computations of $\alpha_{i}, i\in\{0,1,2,3\}$, we have
$$
\alpha_{4}=\frac{1}{\alpha_{3}-1}=\frac{2t+3d-1+\sqrt{D(d,t)}}{3t+10},
$$
and we are interested in positive integer solutions of the inequality $2-\alpha_{4}>0$. Simple computations reveals that
\begin{align*}
2-\alpha_{4}&=\frac{4t-3d+21-\sqrt{D(d,t)}}{3t+10}>0\quad \Longleftrightarrow \quad 4t-3d+21> \sqrt{D(d,t)}\\
            &\Longleftrightarrow 4t-3d+21 >0 \wedge (4t-3d+21)^2>D(d,t)\\
            &\Longleftrightarrow 4t-3d+21 >0 \wedge (41-16d)(3t+10)>0\\
            &\Longleftrightarrow d\in\{1,2\} \wedge t\in\N_{+}
\end{align*}
and hence the result.
\end{proof}

\begin{rem}
{\rm Let us note that the expression for $D$ is reducible (as a polynomial in $\Z[d,t]$) if and only if the discriminant of $D$ with respect to $d$ or $t$ is a square. We have
$$
\op{Disc}_{d}(D)=4 (10 + 3t),\quad \op{Disc}_{t}(D)=41-16d.
$$
In consequence, $\op{Disc}_{d}(D)=\square$ if and only if $t=(u^2-10)/3$ for some $u\equiv 1,2\pmod{3}$ and $|u|\geq 4$. In this case we have
$$
D(d,t)=\frac{1}{9}(9d+4u^2-u-23)(9d+4u^2+u-23).
$$
Moreover, we have $\op{Disc}_{t}(D)=\square$ for $d\in\N_{+}$ if and only if $d=1$ or $d=2$. Then $D(1,t)=(t+2)(16t+37), D(2,t)=(t+3)(16 t+45)$.
}
\end{rem}

Let us recall (the simple form of) the result of Iwaniec concerning the existence of prime values taken by inhomogeneous quadratic forms \cite{Iwa}.

\begin{thm}[Theorem 1 in \cite{Iwa}]\label{iwa}
Let
$$
P(x,y)=ax^2+bxy+cy^2+ex+fy+g\in\Z[x,y],
$$
$\op{deg}P=2, (a,b,c,e,f,g)=1$, $P(x,y)$ be irreducible in $\Q[x,y]$, represents arbitrarily large odd numbers and depend essentially on two variables. Then the set of prime numbers represented by the polynomial $P(x,y)$ is infinite.
\end{thm}

\begin{rem}
{\rm In fact, Iwaniec proved not only the infinitude of the set of prime numbers represented by the polynomial $P$ but also obtained the lower bound
$$
N\log^{-3/2}N\ll\sum_{p=P(x,y)\leq N}1.
$$}
\end{rem}

Using Theorem \ref{three1} together with Iwaniec result we get the following.

\begin{cor}
There are infinitely many prime numbers $p$ such that beginning of the periodic part of $\sqrt{p}$ contains three consecutive 1's. Moreover, for given $N$, there are at least $N\log^{-3/2}N$ prime numbers $p<N$ such that $I_{p,3}=(1,1,1)$ and $I_{p,4}\neq (1,1,1,1)$.
\end{cor}
\begin{proof}
The inhomogeneous quadratic form $D(d,t)$ is: irreducible, represents sufficiently large odd integers (due to the identity $D(d,1)=9d^2+58d+92$) and depends essentially on two variables (due to the fact that the discriminant $\op{Disc}_{d}D(d,t)=4(3t+10)$ is non-constant). From Theorem \ref{iwa} we get the result.
\end{proof}

We proved that there are infinitely many prime numbers $p$ such that $I_{p,3}=(1,1,1)$. However, we are unable to prove anything (unconditionally) about the possible values of the period for $\sqrt{D(d,t)}$, where the value $D(d,t)$ is a prime number. From Lemma \ref{simplelem} we see that under Hypothesis H, we can have the smallest possible period $T_{p}=5$. A question arises whether we can prove the existence of $m$ such that there are infinitely many $p$'s with $T_{p}=m$ and $I_{p,3}=(1,1,1)$ without assumption of Hypothesis H? In order to do that we can try to use our parametric family of $D$'s.
From Lemma \ref{simplelem} and  Theorem \ref{three1} we need to have $m\geq 8$. However, if $T_{D(d,t)}=8$ then Theorem \ref{three1} implies $D=(9u-3x-1)((3x+4)^2u-(x+1)(3x^2+6x+2)$ with $u>\frac{1}{3}(x+2)$. Then, $D$ is a composite integer, as $9u-3x-1>5$ and $(3x+4)^2u-(x+1)(3x^2+6x+2)>3$.

Let us try $m=9$ and observe that $T_{D(d,t)}=9$ if and only if
$$
\sqrt{D(d,t)}=[4t+3d+5;\overline{1,1,1,x,x,1,1,1,2(4t+3d+5)}]
$$
for certain $x\in\N_{+}$. Standard computation reveals that
$$
\sqrt{D(d,t)}=\sqrt{\frac{Q(d,t)}{9x^2+12x+13}},
$$
where
$$
Q(d,t)=(9x^2+12x+13)(4t+3d)^2+2(4t+3d)(51x^2+67x+73)+289x^2+374x+410.
$$
In order to find required solutions we consider the equation $(9x^2+12x+13)D(d,t)=Q(d,t)$ in positive integers $x,d,t$. We have
$$
(9x^2+12x+13)D(d,t)-Q(d,t)=0 \Longleftrightarrow 2(3x+2)d-(3x^2-4x-1)t=10x^2+2x+7.
$$
The above Diophantine equation has solution in integers if and only if $x=2n$ for some $n\in\N_{+}$. Indeed, if $x\equiv 1\pmod{2}$ then $2|\gcd(2(3x+2),3x^2-4x-1)$ and $2\not | 10x^2+2x+7$. On the other hand, if $x=2n$ then it is easy to check the equality $\gcd(2(3x+2),3x^2-4x-1)=1$. Moreover, the pair
$$
d_{0}=4n^3(592n-457),\quad t_{0}=2368n^3+540n^2-52n+7
$$
is a solution of our equation. As a consequence we obtain full description of integer solutions in the following form
$$
d=(12 n^2-8 n-1)u+d_{0},\quad t=4(3n+1)u+t_{0},\quad u\in\N_{+}.
$$
With $d, t$ obtained in this way we get
$$
D(d,t)=P_{0}(n)u^2+P_{1}(n)u+P_{2}(n)=:F(n,u),
$$
where
\begin{align*}
P_{0}(n)&=\left(36n^2+24n+13\right)^2,\\
P_{1}(n)&=2\left(255744n^6+314064n^5+265824n^4+96196n^3+24300n^2-1898n+437\right),\\
P_{2}(n)&=2(25233408n^8+28330752n^7+23296712n^6+7136448n^5+\\
        &\hskip 2 cm 1742464n^4-315412n^3+94262n^2-6994n+565).
\end{align*}
One can easily check that for all $n\in\N_{+}$ we have $\gcd(P_{0}(n),P_{1}(n),P_{2}(n))=1$ and the polynomial $F$ is irreducible over $\Q[n,u]$. Note that $\op{Disc}_{u}(F(n,u))=-4$. A question arises whether the polynomial $F$ (in two variables $n, u$) represents infinitely many prime numbers. It is very likely that the hypothetical proof of this fact should be easier then the proof of any instances of Hypothesis H. This would imply the existence of infinitely many prime numbers $p$ with $T_{p}=9$ and $I_{p,3}=(1,1,1)$.

\section{Generalization of Cassini identity and some applications of Hypothesis H}\label{section3}

The following section is devoted to conditional results on frequency of appearing of number $1$ in the period of continued fraction of square root of prime numbers. In order to state the fundamental result of this section we need to introduce some notation. For given $n\in\N_{+}$ and $i, j\in\N_{+}$ satisfying the condition $i\leq j\leq n$ we will write
\begin{align*}
X_{n}&=(x_{1},\ldots,x_{n}),\\
X_{i,j}&=(x_{i}, x_{i+1}, \ldots, x_{n-1}, x_{n}, x_{n-1}, \ldots, x_{j+1}, x_{j}),\\
Y_{i,j}&=(x_{i}, x_{i+1}, \ldots, x_{n-1}, x_{n}, x_{n}, x_{n-1}, \ldots, x_{j+1}, x_{j}).
\end{align*}
In order to prove the main theoretical result of this section we will need the following.

\begin{thm}\label{Cassinilike}
For each $n\in\N_+$ we have
\begin{align*}
q_{2n-2}(X_{1,2})^2-q_{2n-1}(X_{1,1})q_{2n-3}(X_{2,2})=&1,\\
q_{2n-1}(Y_{1,2})^2-q_{2n}(Y_{1,1})q_{2n-2}(Y_{2,2})=&-1.
\end{align*}
\end{thm}
\begin{proof}
Because the proofs of the identities are analogous, we only present the proof of the first one. We proceed by induction on $n\in\N_+$. For $n=1$ we have $q_{-1}=0$, $q_0=1$, $q_1(x_1)=x_1$. Thus $q_0^2-q_1(x_1)q_{-1}=1$.

Assume now that $n>1$. Then, we use the identities of the form $q_k(x_1,...,x_k)=q_k(x_k,...,x_1)$, $q_k(X_{k})=x_1q_{k-1}(x_2,...,x_k)+q_{k-2}(x_3,...x_k)$ and $q_k(x_1,...,x_k)=x_kq_{k-1}(x_1,...,x_{k-1})+q_{k-2}(x_1,...x_{k-2})$, $k\in\N_+$, to obtain the following chain of equalities:
\begin{align*}
&q_{2n-2}(X_{1,2})^2-q_{2n-1}(X_{1,1})q_{2n-3}(X_{2,2})\\
&=(x_1q_{2n-3}(X_{2,2})+q_{2n-4}(X_{3,2}))^2-(x_1q_{2n-2}(X_{1,2})+q_{2n-3}(X_{1,3}))q_{2n-3}(X_{2,2})\\
& =(x_1q_{2n-3}(X_{2,2})+q_{2n-4}(X_{2,3}))^2-[x_1(x_1q_{2n-3}(X_{2,2})+q_{2n-4}(X_{3,2}))\\
&\quad +x_1q_{2n-4}(X_{2,3})+q_{2n-5}(X_{3,3})]q_{2n-3}(X_{2,2})\\
& =x_1^2q_{2n-3}(X_{2,2})^2+2x_1q_{2n-3}(X_{2,2})q_{2n-4}(X_{2,3})+q_{2n-4}(X_{2,3})^2-[x_1^2q_{2n-3}(X_{2,2})\\
& \quad +2x_1q_{2n-4}(X_{2,3})+q_{2n-5}(X_{3,3})]q_{2n-3}(X_{2,2})
\end{align*}
\begin{align*}
& =x_1^2q_{2n-3}(X_{2,2})^2+2x_1q_{2n-3}(X_{2,2})q_{2n-4}(X_{2,3})+q_{2n-4}(X_{2,3})^2-x_1^2q_{2n-3}(X_{2,2})^2\\
& \quad -2x_1q_{2n-4}(X_{2,3})q_{2n-3}(X_{2,2})-q_{2n-5}(X_{3,3})q_{2n-3}(X_{2,2})\\
& =q_{2n-4}(X_{2,3})^2-q_{2n-5}(X_{3,3})q_{2n-3}(X_{2,2})=1,
\end{align*}
by induction hypothesis for $n-1$.
\end{proof}

\begin{rem}
{\rm The above formulas (under specialization $x_1=...=x_n=1$) can be seen as a generalization of classical Cassini identities for Fibonacci numbers, i.e., $F_{n-1}F_{n+1}-F_{n}^{2}=(-1)^{n}$. }
\end{rem}

We are ready to prove the following.

\begin{thm}\label{redpoly}
Let $n\in\N_{+}$ and define the quadratic polynomials $F_{n}(a,X_{n}),\linebreak G_{n}(a,X_n)\in\Q(X_{n})[a]$ in the following way
\begin{equation*}
F_{n}(a,X_{n})=[a; \overline{X_{1,1}, 2a}]^2,\quad G_{n}(a,X_n)=[a; \overline{Y_{1,1}, 2a}]^2.
\end{equation*}
Then we have the following equalities
\begin{align*}
F_{n}(a,X_{n})&=\left(a+\frac{q_{2n-2}(X_{1,2})+1}{q_{2n-1}(X_{1,1})}\right)\left(a+\frac{q_{2n-2}(X_{1,2})-1}{q_{2n-1}(X_{1,1})}\right),\\
G_{n}(a,X_{n})&=\left(a+\frac{q_{2n-1}(Y_{1,2})}{q_{2n}(Y_{1,1})}\right)^2+\frac{1}{q_{2n}(Y_{1,1})^2}.
\end{align*}
In particular, for any given $X_{n}\in\N_{+}^{n}$, the polynomial $F_{n}(a,X_{n})$ takes only finitely many prime values for $a\in\N_{+}$.
\end{thm}
\begin{proof}
Because the proofs of the equalities are analogous, we only present the proof of the first one. Let us put $\theta=[a;\overline{X_{1,1},2a}]=[a;X_{1,1},\theta+a]$. Then, $\theta$ is a positive root of the quadratic equation
\begin{equation*}
\theta=\frac{(\theta+a)p_{2n-1}(a,X_{1,1})+p_{2n-2}(a,X_{1,2})}{(\theta+a)q_{2n-1}(X_{1,1})+q_{2n-2}(X_{1,2})}
\end{equation*}
or equivalently,
\begin{equation}\label{quadeq}
\begin{split}
& \theta^2q_{2n-1}(X_{1,1})+\theta[aq_{2n-1}(X_{1,1})+q_{2n-2}(X_{2,1})-p_{2n-1}(a,X_{1,1})]\\
& =ap_{2n-1}(a,X_{1,1})+p_{2n-2}(a,X_{1,2}).
\end{split}
\end{equation}
From the general theory of continued fractions we know that
\begin{equation*}
p_{2n-1}(a,X_{1,1})=aq_{2n-1}(X_{1,1})+q_{2n-2}(X_{2,1})=aq_{2n-1}(X_{1,1})+q_{2n-2}(X_{1,2})
\end{equation*}
and
$$p_{2n-2}(a,X_{1,2})=aq_{2n-2}(X_{1,2})+q_{2n-3}(X_{2,2}).$$
Hence, the equation (\ref{quadeq}) takes the form
\begin{equation*}
\theta^2q_{2n-1}(X_{1,1})=a^2q_{2n-1}(X_{1,1})+2aq_{2n-2}(X_{1,2})+q_{2n-3}(X_{2,2}).
\end{equation*}
We thus obtain
\begin{align*}
F_n(a,X)&=\theta^2=a^2+2a\frac{q_{2n-2}(X_{1,2})}{q_{2n-1}(X_{1,1})}+\frac{q_{2n-3}(X_{2,2})}{q_{2n-1}(X_{1,1})}\\
        &=\left(a+\frac{q_{2n-2}(X_{1,2})}{q_{2n-1}(X_{1,1})}\right)^2-\frac{q_{2n-2}(X_{1,2})^2-q_{2n-1}(X_{1,1})q_{2n-3}(X_{2,2})}{q_{2n-1}(X_{1,1})^2}\\
        &=\left(a+\frac{q_{2n-2}(X_{1,2})}{q_{2n-1}(X_{1,1})}\right)^2-\frac{1}{q_{2n-1}(X_{1,1})^2}\\
        &=\left(a+\frac{q_{2n-2}(X_{1,2})+1}{q_{2n-1}(X_{1,1})}\right)\left(a+\frac{q_{2n-2}(X_{1,2})-1}{q_{2n-1}(X_{1,1})}\right),
\end{align*}
where the equality on the third line holds by Lemma \ref{Cassinilike}.
\end{proof}

\begin{prop}\label{irred}
Under assumption of Hypothesis H, for any given $n\in\N$ not congruent to $1$ modulo $3$ and $X_{n}=(x_1,...,x_n)\in\N_{+}^{n}$ such that $x_1,...,x_n$ are all odd, the polynomial $G_{n}(a,X_{n})$ takes infinitely many prime values for $a\in\N_{+}$.
\end{prop}

\begin{proof}
At first, let us note that if $x_1,...,x_n$ are all odd integers, then $q_n(X_n)\equiv F_{n+1}\pmod{2}$, where $F_n$ denotes $n$-th Fibonacci number. Indeed, $q_{-1}=F_0$, $q_0=F_1$ and $q_k(X_k)=x_kq_{k-1}(X_{k-1})+q_{k-2}(X_{k-2})\equiv F_k+F_{k-1}=F_{k+1}\pmod{2}$ for $k\in\N_+$.

From Theorem \ref{redpoly} we know that $G_n(a,X_n)=a^2+\frac{2aq_{2n-1}(X_{1,2})+q_{2n-2}(X_{2,2})}{q_{2n}(X_{1,1})}$. There exists $a_0\in\N$ such that $\frac{2a_0q_{2n-1}(X_{1,2})+q_{2n-2}(X_{2,2})}{q_{2n}(X_{1,1})}\in\Z$. Indeed, the congruence $2aq_{2n-1}(X_{1,2})+q_{2n-2}(X_{2,2})\equiv 0\pmod{q_{2n}(X_{1,1})}$ has a solution in $a\in\Z$ as $\gcd(q_{2n-1}(X_{1,2}),q_{2n}(X_{1,1}))=1$ and $q_{2n}(X_{1,1})\equiv F_{2n+1}\equiv 1\pmod{2}$. Then $\frac{2aq_{2n-1}(X_{1,2})+q_{2n-2}(X_{2,2})}{q_{2n}(X_{1,1})}\in\Z$ if and only if $a=a_0+tq_{2n}(X_{1,1})$ for some $t\in\Z$. Thus
\begin{align*}
\tilde{G}_n(t,X_n):&=G_n(a_0+tq_{2n}(X_{1,1}),X_n)\\
                   &=(a_0+tq_{2n}(X_{1,1}))^2+2tq_{2n-1}(X_{1,2})+\frac{2a_0q_{2n-1}(X_{1,2})+q_{2n-2}(X_{2,2})}{q_{2n}(X_{1,1})}
\end{align*}
and we see that for each $t\in\Z$ the number $\tilde{G}_n(t,X_n)$ is an integer. Because $2\nmid q_{2n}(X_{1,1})$, the polynomial $\tilde{G}_n(t,X_n)$ (with respect to the variable $t$) attains odd values. Since its degree is $2$ and the coefficients near $t^2$ and $t$ are coprime, for each odd prime number $p$ there exists a value $t\in\Z$ such that $p\nmid\tilde{G}_n(t,X_n)$. Moreover, the leading coefficient of $\tilde{G}_n(t,X_n)$ is positive. Hence, assuming validity of Hypothesis H, we conclude that $G_{n}(a,X_{n})$ takes infinitely many prime values for $a\in\N_{+}$.
\end{proof}

Let us denote
$$
L_{1}(m):=\mbox{the number of}\; 1's\;\mbox{in the period of continued fraction expansion of}\;\sqrt{p_{m}}.
$$

It is an interesting question whether the number $L_{1}(m)$ can be somehow compared with $T_{p_{m}}$. It is clear that $0\leq L_{1}(m)/T_{p_{m}}<1$.  Numerical calculations for $m\leq 10^6$ suggest that the set $\{L_{1}(m)/T_{p_{m}}:\;m\in\N_{+}\}$ is dense in $(0,1)$. Unfortunately, we were unable to prove such a statement unconditionally.



\begin{thm}\label{dense}
Under assumption of Hypothesis H, the set
$$
\left\{\frac{L_{1}(m)}{T_{p_{m}}}:\;m\in\N_{+}\right\}
$$
is dense in $[0,1]$.
\end{thm}

Before we show the above theorem, let us put
\begin{align*}
\cal{L}_{i}:=&\{p_{m}:\;L_{1}(m)=i\}\\
\cal{L}_{i,n}:=&\{p_{m}:\;L_{1}(m)=i, T_{p_m}=n\}
\end{align*}
for each $i,n\in\N$ and notice the following consequence of Proposition \ref{irred}.

\begin{cor}\label{infty}
If Hypothesis H is true, then the set $\cal{L}_{2i,2n+1}$ is infinite for each $i,n\in\N$ such that $n\not\equiv 1\pmod{3}$. In particular, the set $\cal{L}_{2i}$ is infinite for any $i\in\N$.
\end{cor}

\begin{proof}
For fixed $i,n\in\N$, take $X_n=(x_1,...,x_n)\in\N_+^n$ such that all the values $x_j$ are odd and exactly $i$ of them are equal to $1$. Then the polynomial $G_n(a,X_n)$ represents infinitely many prime numbers and in the period of continued fraction expansion of $\sqrt{G_n(a,X_n)}$ we have exactly $2i$ occurrences of $1$.
\end{proof}

At this moment, the proof of Theorem \ref{dense} is easy.

\begin{proof}[Proof of Theorem \ref{dense}]
Corollary \ref{infty} implies that
$$
\left\{\frac{2i}{2n+1}:\;i,n\in\N, i\leq n, n\not\equiv 1\pmod{3}\right\}\subset\left\{\frac{L_{1}(m)}{T_{p_{m}}}:\;m\in\N_{+}\right\}.
$$
Then, for each $u\in [0,1]$ we have
$$0\leq\inf_{n\in\N, n\not\equiv 1(\bmod{3})}\inf_{i\in\{0,...,n\}}\left|u-\frac{2i}{2n+1}\right|\leq\inf_{n\in\N, n\not\equiv 1(\bmod{3})}\frac{1}{2n+1}=0.$$
This means that $\left\{\frac{2i}{2n+1}:\;i,n\in\N, n\not\equiv 1\pmod{3}\right\}$ is dense in $[0,1]$. As a result, the set $\left\{\frac{L_{1}(m)}{T_{p_{m}}}:\;m\in\N_{+}\right\}$ is dense in $[0,1]$.
\end{proof}

\section{Applications of equidistribution of the sequence $(\{\sqrt{p_{m}}\})_{m\in\N_{+}}$}\label{section4}

Let us recall that a sequence $(u_n)_{n\in\N_+}$ of real numbers is equidistributed modulo $1$ if for each $1\geq\alpha<\beta\leq 1$ $\lim_{N\rightarrow +\infty}\frac{1}{N}\#\{n\in\{1,...,N\}: \{u_n\}\in (\alpha,\beta)\}=\beta-\alpha$, where $\{x\}$ means the fractional part of the number $x$. Then, we state the following result (see \cite[Exercise 2, p. 348 and a comment after Theorem 21.3]{IwaKow}).

\begin{thm}\label{equidist}
The sequence $(\sqrt{p_n})_{n\in\N_+}$ is equidistributed modulo $1$.
\end{thm}

For $k\in\N_{+}$ we denote the set of primes satisfying the conditions $I_{p,k}=(1,\ldots,1)$ and $I_{p,k+1}\neq (1,\dots,1)$ by $\cal{A}_{k}$. We will show that for each $k\in\N_{+}$ the set $\cal{A}_{k}$ is not only non-empty, but even infinite. Actually, using the result on equidistribution modulo $1$ of square roots of primes, we will prove much stronger fact. Indeed, the above theorem allows us to show a general result on infinitude of the set of prime numbers with prescribed first elements of continued fractions of their square roots. Before we state it, we introduce the notion of relative asymptotic density of a given subset $A$ of prime numbers in the set of all prime numbers $\bbb{P}$ (or relative asymptotic density, for short):
$$d_{\bbb{P}}(A)=\lim_{N\rightarrow +\infty}\frac{1}{N}\#\{n\in\{1,...,N\}: p_n\in A\}=\lim_{t\rightarrow +\infty}\frac{1}{\pi(t)}\#\{p\in\bbb{P}: p\in A, p\leq t\}.$$

In order to simplify the notation, instead of writing $(a_{1},\ldots, a_{k})$, where $k\in\N_{+}$ and $a_{1},\ldots, a_{k}\in\N_{+}$ are given, we will write $\textbf{a}_{k}$.

\begin{thm}
Let $\bbb{P}_{\emph{\textbf{a}}_{k}}$ be the set of all prime numbers $p$ such that $I_{p,k}=(\textcolor{red}{\emph{\textbf{a}}}_{k})$. Then the value $d_{\bbb{P}}(\bbb{P}_{\emph{\textbf{a}}_{k}})$ exists and it is equal to $\frac{1}{(q_k(\emph{\textbf{a}}_{k})+q_{k-1}(\emph{\textbf{a}}_{k-1}))q_k(\emph{\textbf{a}}_{k})}$. In particular, the set $\bbb{P}_{\emph{\textbf{a}}_{k}}$ is infinite.
\end{thm}

\begin{proof}
By Theorem \ref{equidist} we know that the sequence of square roots of all consecutive prime numbers is equidistributed modulo $1$. Thus, $d_{\bbb{P}}(\bbb{P}_{\textbf{a}_{k}})$ is equal to the length of the interval $\cal{J}_{\textbf{a}_{k}}$ of all non-negative real numbers $<1$ such that the first $k$ elements in their continued fractions are $\textbf{a}_{k}$. Since
$$\cal{J}_{\textbf{a}_{k}}=\left\{\frac{\theta p_k(\textbf{a}_{k})+p_{k-1}(\textbf{a}_{k-1})}{\theta q_k(\textbf{a}_{k})+q_{k-1}(\textbf{a}_{k-1})}: \theta\in [1,+\infty)\right\},$$
hence the length of $\cal{J}_{\textbf{a}_{k}}$ is equal to
\begin{align*} \left|\frac{p_k(\textbf{a}_{k})+p_{k-1}(\textbf{a}_{k-1})}{q_k(\textbf{a}_{k})+q_{k-1}(\textbf{a}_{k-1})}-\frac{p_k(\textbf{a}_{k})}{q_k(\textbf{a}_{k})}\right|&=\frac{|p_{k-1}(\textbf{a}_{k-1})q_k(\textbf{a}_{k})-p_k(\textbf{a}_{k})q_{k-1}(\textbf{a}_{k-1})|}{(q_k(\textbf{a}_{k})+q_{k-1}(\textbf{a}_{k-1}))q_k(\textbf{a}_{k})}\\
&=\frac{1}{(q_k(\textbf{a}_{k})+q_{k-1}(\textbf{a}_{k-1}))q_k(\textbf{a}_{k})}.
\end{align*}
\end{proof}

As a result, we can compute the relative asymptotic density of the set $\cal{A}_k$ in $\bbb{P}$ and conclude infinitude of this set.

\begin{cor}
For each $k\in\N_+$ we have $d_{\bbb{P}}(\cal{A}_k)=\frac{1}{F_{k+3}F_{k+1}}$, where $F_k$ denotes $k$-th Fibonacci number. In particular, the set $\cal{A}_k$ is infinite.
\end{cor}

\begin{proof}
By Theorem \ref{equidist}, $d_{\bbb{P}}(\cal{A}_k)$ is equal to the Lebesgue measure of the set of all non-negative real numbers $<1$ such that the first $k$ elements in their continued fractions are $1$ and $k+1$-st element is other than $1$. This set is $\cal{J}_{\textbf{1}^k}\bs\cal{J}_{\textbf{1}^{k+1}}$, where $\textbf{1}^k:=(\underset{k\mbox{\tiny times}}{\underbrace{1,...,1}})$. From the definition of polynomials $q_n(X_n)$ we see that $q_n(\textbf{1}^n)=F_{n+1}$ for each $n\in\N$. Hence, using the previous theorem, we obtain
\begin{align*}
d_{\bbb{P}}(\cal{A}_k)&=\frac{1}{(q_k(\textbf{1}^{k})+q_{k-1}(\textbf{1}^{k-1}))q_k(\textbf{1}^{k})}-\frac{1}{(q_{k+1}(\textbf{1}^{k+1})+q_{k}(\textbf{1}^{k}))q_{k+1}(\textbf{1}^{k+1})}\\
&=\frac{1}{q_{k+1}(\textbf{1}^{k+1})q_{k}(\textbf{1}^{k})}-\frac{1}{q_{k+2}(\textbf{1}^{k})q_{k+1}(\textbf{1}^{k-1})}=\frac{1}{F_{k+2}F_{k+1}}-\frac{1}{F_{k+3}F_{k+2}}\\
&=\frac{F_{k+3}-F_{k+1}}{F_{k+3}F_{k+2}F_{k+1}}=\frac{F_{k+2}}{F_{k+3}F_{k+2}F_{k+1}}=\frac{1}{F_{k+3}F_{k+1}}.
\end{align*}
\end{proof}

We are also able to estimate relative asymptotic densities of the sets of prime numbers with prescribed not necessarily first consecutive elements of continued fraction expansions of their square roots. Let us denote the $n$-th element of continued fraction expansion of a real number $x$ as $a_n(x)$. Then, define
\begin{align*}
\bbb{P}\left(\begin{array}{ccc}
n_1 & ... & n_s\\
a_1 & ... & a_s
\end{array}\right)=&\{p\in\bbb{P}:\; a_{n_j}(\sqrt{p})=a_j\mbox{ for each }j\in\{1,...,s\}\}\\
E\left(\begin{array}{ccc}
n_1 & ... & n_s\\
a_1 & ... & a_s
\end{array}\right)=&\{x\in [0,1]:\; a_{n_j}(x)=a_j\mbox{ for each }j\in\{1,...,s\}\}.
\end{align*}

\begin{thm}\label{Khinestim}
For each $n,a\in\N_+$ we have $\frac{1}{3a^2}<d_{\bbb{P}}\left(\bbb{P}{n\choose a}\right)<\frac{2}{a^2}$. Moreover, there exist absolute constants $C,\lambda >0$ such that for each $s\in\N_+$ and $n_1,...,n_s,a_1,...,a_s\in\N_+$ we have
\begin{align*}
&\prod_{j=1}^s\left(\frac{\ln\left(1+\frac{1}{a_j(a_j+2)}\right)}{\ln 2}-\frac{C}{a_j(a_j+1)}e^{-\lambda\sqrt{n_j-n_{j-1}-1}}\right)<d_{\bbb{P}}\left(\bbb{P}\left(\begin{array}{ccc}
n_1 & ... & n_s\\
a_1 & ... & a_s
\end{array}\right)\right)\\
<&\prod_{j=1}^s\left(\frac{\ln\left(1+\frac{1}{a_j(a_j+2)}\right)}{\ln 2}+\frac{C}{a_j(a_j+1)}e^{-\lambda\sqrt{n_j-n_{j-1}-1}}\right),
\end{align*}
where we put $n_0=0$.
\end{thm}

\begin{proof}
From Theorem \ref{equidist} we know that the relative asymptotic densities of the sets $\bbb{P}{n\choose a}$ and $\bbb{P}\left(\begin{array}{ccc}
n_1 & ... & n_s\\
a_1 & ... & a_s
\end{array}\right)$ are equal to the Lebesgue measures of the sets $E{n\choose a}$ and\linebreak $E\left(\begin{array}{ccc}
n_1 & ... & n_s\\
a_1 & ... & a_s
\end{array}\right)$, respectively. The estimations of the Lebesgue measures of the latter sets follow from \cite[p. 60 and Theorem 34, p. 85-86]{Khi}.
\end{proof}

As a consequence of the above theorem, we conclude the fact that the set $\cal{L}_0$ of all the prime numbers without $1$ in the period of the continued fraction expansion has relative asymptotic density equal to $0$.

\begin{cor}\label{L0}
We have $d_{\bbb{P}}(\cal{L}_0)=0$.
\end{cor}

\begin{proof}
For each positive integers $s$ and $n_1<...<n_s$, the set $\cal{L}_0$ is a subset of the union $\bigcup_{a_1,...,a_s\geq 2}\bbb{P}\left(\begin{array}{ccc}
n_1 & ... & n_s\\
a_1 & ... & a_s
\end{array}\right)$. Thus, by monotonicity of relative asymptotic density and the upper bound for the value $d_{\bbb{P}}\left(\bbb{P}\left(\begin{array}{ccc}
n_1 & ... & n_s\\
a_1 & ... & a_s
\end{array}\right)\right)$, we obtain the following.
\begin{equation}\label{upperbound}
\begin{aligned}
& d_{\bbb{P}}(\cal{L}_0)\leq\sum_{a_1,...,a_s\geq 2}d_{\bbb{P}}\left(\bbb{P}\left(\begin{array}{ccc}
n_1 & ... & n_s\\
a_1 & ... & a_s
\end{array}\right)\right)\\
& <\sum_{a_1,...,a_s\geq 2}\prod_{j=1}^s\left(\frac{\ln\left(1+\frac{1}{a_j(a_j+2)}\right)}{\ln 2}+\frac{C}{a_j(a_j+1)}e^{-\lambda\sqrt{n_j-n_{j-1}-1}}\right)\\
& =\prod_{j=1}^s\sum_{a_j=2}^{+\infty}\left(\frac{\ln\left(1+\frac{1}{a_j(a_j+2)}\right)}{\ln 2}+\frac{C}{a_j(a_j+1)}e^{-\lambda\sqrt{n_j-n_{j-1}-1}}\right),
\end{aligned}
\end{equation}
where $C,\lambda>0$ are absolute constants and we put $n_0=0$. Let $k\in\N_+$ be so large that $$\frac{\ln\frac{3}{2}}{\ln 2}+\frac{C}{2}e^{-\lambda\sqrt{k}}<1.$$ Put $n_j=j(k+1)$ for $j\in\{1,...,s\}$. Then, the estimation (\ref{upperbound}) takes the form:
\begin{equation}
\begin{aligned}
& d_{\bbb{P}}(\cal{L}_0)<\prod_{j=1}^s\sum_{a_j=2}^{+\infty}\left(\frac{\ln\left(1+\frac{1}{a_j(a_j+2)}\right)}{\ln 2}+\frac{C}{a_j(a_j+1)}e^{-\lambda\sqrt{j(k+1)-(j-1)(k+1)-1}}\right)\\
& =\prod_{j=1}^s\left[\frac{1}{\ln 2}\ln\prod_{a_j=2}^{+\infty}\left(\frac{(a_j+1)^2}{a_j(a_j+2)}\right)+Ce^{-\lambda\sqrt{k}}\sum_{a_j=2}^{+\infty}\left(\frac{1}{a_j}-\frac{1}{a_j+1}\right)\right]\\
& =\left(\frac{\ln\frac{3}{2}}{\ln 2}+\frac{C}{2}e^{-\lambda\sqrt{k}}\right)^s.
\end{aligned}
\end{equation}
Since $s$ was taken arbitrarily and $\frac{\ln\frac{3}{2}}{\ln 2}+\frac{C}{2}e^{-\lambda\sqrt{k}}<1$, we deduce the equality $d_{\bbb{P}}(\cal{L}_0)=0$.
\end{proof}

\begin{rem}
{\rm In fact, by the same way we can prove that for each positive integer $c$ the set of prime numbers without $c$ as a digit in the period of the continued fraction expansion of their square roots has relative asymptotic density equal to $0$.}
\end{rem}

\section{Remarks, numerical results, questions and conjectures}

In the light of our results, for given $k\in\N_{+}$ it is natural to compute the smallest prime $p$ satisfying $I_{p,k}=(1,\ldots,1)$ and $I_{p,k+1}\neq (1,\ldots,1,1)$, i.e., the number $P_{k}=\op{min}\cal{A}_{k}$. We performed numerical search for elements of the set $\cal{A}_{k}$ for $k\leq 20$. Let us describe the strategy we used in our search. Instead of computing the full continued fraction expansion of $\sqrt{p}, p\in\mathbb{P}$, we computed only the first 20 terms in the periodic part of the expansion, i.e., $\sqrt{p}=[a_{0};a_{1},\ldots,a_{20},x']$. Next we checked whether $(a_{1},\ldots,a_{k})=(1,\ldots,1)$ for $k=1,2,\ldots,20$. With this approach we were able to compute the number of elements in the set $\cal{A}_{k}\cap [1,p_{10^7}]$ and its smallest element. In this range there was no elements of the set $\cal{A}_{k}$ for $k=16,\ldots,20$. We thus extended the search up to $2\cdot 10^8$th prime and found the elements of $P_{k}$ for the remaining values of $k$.

In the table below we present results of our search of the prime number $P_{k}$, together with the length of the period $T_{P_{k}}$.
\begin{equation*}
\begin{array}{|l|ll||l|ll|}
\hline
 k & P_{k} & T_{P_{k}}  & k   & P_{k}      & T_{P_{k}}      \\
 \hline
 1 & 3     & 2      & 11  & 111301     & 211          \\
 2 & 31    & 8      & 12  & 1258027    & 1578           \\
 3 & 7     & 4      & 13  & 5325101    & 2067        \\
 4 & 13    & 5      & 14  & 12564317   & 1551        \\
 5 & 3797  & 13     & 15  & 9477889    & 937          \\
 6 & 5273  & 7      & 16  & 47370431   & 6900             \\
 7 & 4987  & 66     & 17  & 709669249  & 58721        \\
 8 & 90371 & 258    & 18  & 1529640443 & 6682      \\
 9 & 79873 & 257    & 19  & 2196104969 & 45875             \\
 10 & 2081 & 11     & 20  & 392143681  & 24087         \\
 \hline
\end{array}
\end{equation*}
\begin{center}Table 1. The smallest prime number $p\in\cal{A}_{k}$ together with the period $T_{p}$.\end{center}
\bigskip
As we can see the behaviour of the function $P_{k}$ is quite irregular (in particular it is not increasing). In the next table we present the number of elements of the set $\cal{A}_{k}\cap [1,p_{10^7}]$.
\begin{equation*}
\begin{array}{|l|l||l|l|}
\hline
    k & |\cal{A}_{k}\cap [1,p_{10^7}]| & k & |\cal{A}_{k}\cap [1,p_{10^7}]| \\
    \hline
    1 & 3333716                        & 11 & 194 \\
    2 & 998469                         & 12 & 69 \\
    3 & 416637                         & 13 & 25 \\
    4 & 154220                         & 14 & 10 \\
    5 & 59424                          & 15 & 5 \\
    6 & 22551                          & 16 & 3 \\
    7 & 8602                           & 17 & 0 \\
    8 & 3278                           & 18 & 0 \\
    9 & 1222                           & 19 & 0 \\
    10 & 481                           & 20 & 0 \\
    \hline
  \end{array}
 \end{equation*}
\begin{center}Table 2. The number of elements in the set $\cal{A}_{k}\cap [1,p_{10^7}], k\in\{1,\ldots,20\}$. \end{center}

We also obtained some numerical data concerning behaviour of the sequence $(T_{p_{m}})_{m\in\N_{+}}$. The main question is how big the number $T_{p_{m}}$ can be. We expect the equality $T_{p_{m}}=O(\sqrt{m}\log m)$. The figures presented below show the behaviour of $T_{p_{m}}$ and the ratio $T_{p_{m}}/\sqrt{m}\log m$ for $m\leq 10^5$.

\begin{figure}[h!]\label{Tpm1} 
       \centering
         \includegraphics[width=3in]{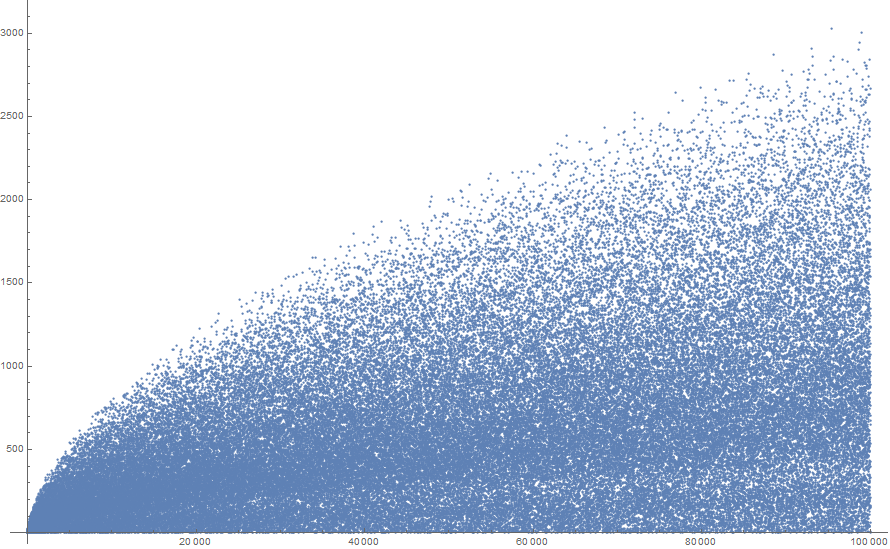}
        \caption{Plot of the function $\protect T_{p_{m}}$}
    \end{figure}

\begin{figure}[h!]\label{Tpm2} 
       \centering
         \includegraphics[width=3in]{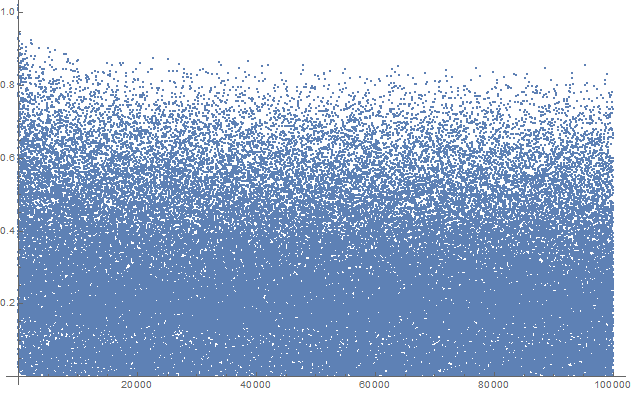}
        \caption{Plot of the function $\protect T_{p_{m}}/\sqrt{m}\log m$}
    \end{figure}
However, numerical computations show that in the range $m\leq 10^7$ there are only two values satisfying $T_{p_{m}}>\sqrt{m}\log m$. There are $m=2$ and $m=4$. This leads us to the following
\begin{ques}
\begin{enumerate}
\item Is the inequality $T_{p_{m}}<\sqrt{m}\log m$ true for all $m\geq 5$?
\item What is the value of the number
$$
\limsup_{m\rightarrow +\infty}\frac{T_{p_{m}}}{\sqrt{m}\log m} ?
$$
\end{enumerate}
\end{ques}

Let us also define:
$$
\cal{W}_{i}:=\{p:\;T_{p}=i\}.
$$

Let us note that there are exactly 35360 values of $i\in\N_{+}$ such that $\cal{W}_{i}\cap\{p_m: m\leq 10^7\}$ is non-empty. Moreover, the smallest ten values of $i\in\N_{+}$ which are not value of $T_{p}$ for $p$ in the considered range are:
$$
31079, 31259, 31399, 31427, 31465, 31621, 31625, 31719, 31754, 31813.
$$

The plot of the number of primes $p_{m}, m\leq 10^7$, such that $T_{p_{m}}=i$ for $i\leq 100$ is given below.

\begin{figure}[htbp] 
       \centering
         \includegraphics[width=3in]{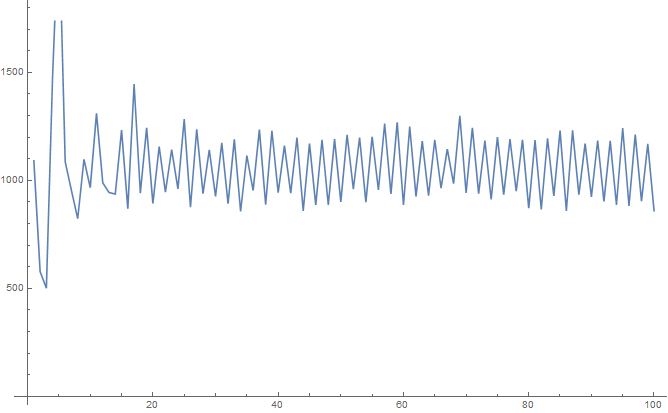}
        \caption{The number of primes $p_{m}, m\leq 10^7$, such that $T_{p_{m}}=i$ for $i\leq 100$}
    \end{figure}

A quick look on the picture presented above suggests the following.

\begin{conj}
For each $i\in \N_{+}$ the limits
$$
\lim_{N\rightarrow +\infty}\frac{\cal{W}_{2i+2}\cap\{p_{1},\ldots,p_{N}\}}{\cal{W}_{2i}\cap\{p_{1},\ldots,p_{N}\}},\quad \lim_{N\rightarrow +\infty}\frac{\cal{W}_{2i+1}\cap\{p_{1},\ldots,p_{N}\}}{\cal{W}_{2i-1}\cap\{p_{1},\ldots,p_{N}\}}
$$
exist and both are equal to 1.
\end{conj}

\bigskip

Let us recall the definition of the set
$$
\cal{L}_{i}:=\{p_{m}:\;L_{1}(m)=i\}.
$$
We took a look on the set $\cal{L}_0$. First of all, we computed all elements in the set $\cal{L}_{0}\cap [1,p_{10^{9}}]$. As we should suppose according to Corollary \ref{L0}, this set is quite small (compared to the set number of elements in the set $\{p_{1},\ldots,p_{10^9}\}$) and equal to
$$
|\cal{L}_{0}\cap [1,p_{10^{9}}]|=25874.
$$
Moreover, we have the equality
$$
\op{max}\{T_{p}:\;p\in \cal{L}_{0}\cap [1,p_{10^{9}}]\}=31.
$$
This number is very small, even when we compare it with the number
$$
\op{max}\{T_{p_{m}}:\;m\leq 10^7\}=40700.
$$
Let
$$
\cal{A}_{0,i}:=\{p\in \cal{L}_{0}:\;T_{p}=i\}.
$$
In Table 3, given below, we present the number, say $A_{i}$, of elements $p\in \cal{A}_{0,i}\cap [1,p_{10^{9}}]$ for $i\in\{1,\ldots, 31\}$.

\begin{equation*}
\begin{array}{|c|llllllllllll|}
\hline
  i     & 1   & 2   & 3   & 5   & 6   & 7   & 9  & 10 & 11 & 13 & 14 & 15 \\
  \hline
  A_{i} &8278 & 4328& 4226& 2696& 2645& 1129& 725& 670& 353& 218& 227& 119\\
  \hline
  \hline
  i     & 17  & 18  & 19  & 21  & 22 & 23 & 25 & 26 & 29 & 30 & 31 & \\
  \hline
  A_{i} & 65  & 67  & 46  & 19  & 21 & 14 & 7  & 12 & 4  & 4  & 1  &\\
  \hline
\end{array}
\end{equation*}
\begin{center}
Table 3. The number of primes in the set $\cal{A}_{0,i}\cap [1,p_{10^{9}}]$ for $i\in\{1,\ldots, 31\}$
\end{center}

In the considered range there is no $p$ such that $T_{p}\equiv 0\pmod{4}$ or $T_{p}=27$, thus, in the table, we omit these values of $i$. In order to compute the values we were interested in, we used a slight modification of the method used in computations of the smallest prime $p$ with $I_{p,k}=(1,\ldots,1)$. This time we first computed the set, say $\cal{B}$, of sequences $I_{p_{m},20}$ without any 1 in it and $m\leq 10^9$. Next, for each $p\in\cal{B}$, we computed the full continued fraction expansion of $\sqrt{p}$, and choose those without 1 in  periodic part. Our computation allows us to compute the smallest element of the set $\cal{A}_{0,i}\cap [1,p_{10^{9}}]$ for $i\leq 31$ not divisible by 4. These numbers together with the continued fraction expansion of its square roots are presented in Table 4.

\begin{equation*}
\begin{array}{|l|l|l|}
\hline
 k & Q_{k} & \mbox{Continued fraction expansion of}\;\sqrt{Q_{k}} \\
\hline
 1 & 2 & [1;\overline{2}] \\
 2 & 11 & [3;\overline{3,6}] \\
 3 & 41 & [6;\overline{2,2,12}] \\
 5 & 89 & [9;\overline{2,3,3,2,18}] \\
 6 & 131 & [11;\overline{2,4,11,4,2,22}] \\
 7 & 1301 & [36;\overline{14,2,2,2,2,14,72}] \\
 9 & 287537 & [536;\overline{4,2,4,2,2,4,2,4,1072}] \\
 10 & 5107 & [71;\overline{2,6,3,4,71,4,3,6,2,142}] \\
 11 & 4649 & [68;\overline{5,2,4,4,27,27,4,4,2,5,136}] \\
 13 & 617801 & [786;\overline{314,2,2,62,2,12,12,2,62,2,2,314,1572}] \\
 14 & 164051 & [405;\overline{31,6,2,4,3,62,405,62,3,4,2,6,31,810}] \\
 15 & 2769149 & [1664;\overline{13,6,2,9,3,16,2,2,16,3,9,2,6,13,3328}] \\
 17 & 40033853 & [6327;\overline{4,3,19,2,3,3,2,3,3,2,3,3,2,19,3,4,12654}] \\
 18 & 975083 & [987;\overline{2,6,4,6,3,8,3,4,987,4,3,8,3,6,4,6,2,1974}] \\
 19 & 20789 & [144;\overline{5,2,3,2,57,4,4,2,11,11,2,4,4,57,2,3,2,5,288}] \\
 21 & 43273913 & [6578;\overline{3,2,3,2,2,10,11,6,3,4,4,3,6,11,10,2,2,3,2,3,13156}] \\
 22 & 1642211 & [1281;\overline{2,20,256,4,51,102,2,512,10,4,1281,4,10,512,2,102,51,4,256,20,2,2562}] \\
 23 & 1799533 & [1341;\overline{2,7,65,3,3,2,6,3,2,5,8,8,5,2,3,6,2,3,3,65,7,2,2682}] \\
 25 & 932718953 & [30540;\overline{2,4,3,2,3,2,2,8,4,12,14,2,2,14,12,4,8,2,2,3,2,3,4,2,61080}] \\
 26 & 104103683 & [10203;\overline{8,4,35,16,600,8,4,1200,8,70,2,16,10203,16,2,70,8,1200,4,8,600,16,35,4,8,20406}]\\
 29 & 47465053 & [6889;\overline{2,21,2,2,6,14,4,7,2,5,4,3,2,5,5,2,3,4,5,2,7,4,14,6,2,2,21,2,13778}] \\
 30 & 102765043 & [10137;\overline{3,4,3,8,100,4,66,132,2,200,4,6,2,6,10137,6,2,6,4,200,2,132,66,4,100,8,3,4,3,20274}] \\
 31 & 10345006913 & [101710;\overline{2,2,5,4,3,6,2,2,2,4,2,7,3,25,10,10,25,3,7,2,4,2,2,2,6,3,4,5,2,2,203420}]\\
 \hline
\end{array}
\end{equation*}
\begin{center} Table 4. The smallest values of $p\in\mathbb{P}$ satisfying $T_{p}=k$ and without 1's in the periodic part $I_{p,T_{p}}$.\end{center}

Based on our computations one can ask the following.

\begin{ques}
Does there exist $p\in \cal{L}_{0}$ satisfying $T_{p}\equiv 0\pmod{4}$?
\end{ques}


\bigskip

In case of the set $\cal{L}_{i}, i\geq 1$, the situation seems to be interesting, too. More precisely, if $i\geq 5$ is odd then the set $\cal{L}_{i}\cap [1,p_{10^{7}}]$ is empty. Moreover,
$$
\cal{L}_{1}\cap [1,p_{10^{7}}]=\{3\},\quad \cal{L}_{3}\cap [1,p_{10^{7}}]=\{7\}.
$$
We also observed that there are 7804 integers $i$ such that there is $q\in\{p_{1},\ldots,p_{10^{7}}\}$ with exactly $i$ occurrences of 1 in the periodic part of the continued fraction expansion of $\sqrt{q}$. The largest number $i$ realizable in this way is 16906. On the other side, the first ten smallest values of $i$ which are not realizable in this way are: $14302, 14774, 14792, 14798, 14826, 14872, 14948, 14960, 15012, 15034$.
In the table below we collect the numbers $|\cal{L}_{i}\cap [1,p_{10^{7}}]|$ for even $i\leq 32$.

\begin{figure}[htbp] 
       \centering
         \includegraphics[width=3in]{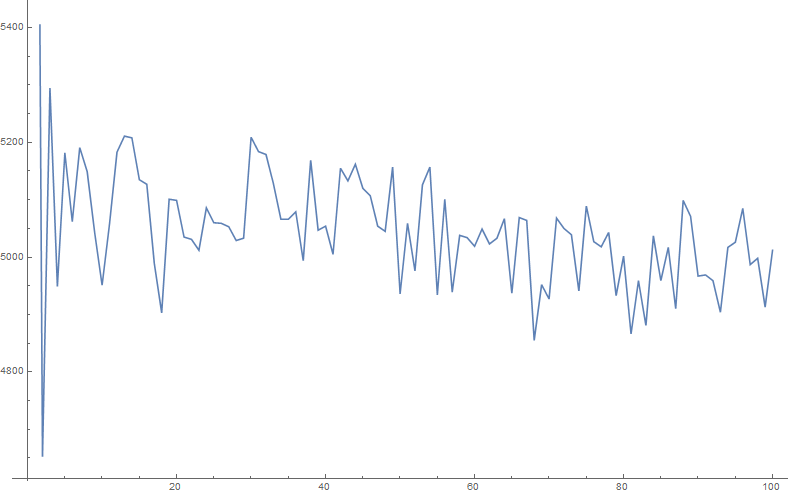}
        \caption{The number of primes $p_{m}, m\leq 10^7$, with exactly $2i$ occurrences of 1 in continued fraction expansion of $\sqrt{p_{m}}$ for $i\leq 100$.}
    \end{figure}

\begin{equation*}
\begin{array}{|c|llllllll|}
\hline
  i                                & 2   & 4   & 6   & 8   & 10  & 12  & 14  & 16  \\
  |\cal{L}_{i}\cap [1,p_{10^{7}}]| & 6814& 4652& 5295& 4949& 5182& 5062& 5191& 5149 \\
  \hline
  \hline
  i                                & 18  & 20  & 22  & 24  & 26  & 28  & 30  & 32 \\
  |\cal{L}_{i}\cap [1,p_{10^{7}}]| & 5043& 4951& 5058& 5183& 5211& 5208& 5135& 5127\\
  \hline
\end{array}
\end{equation*}
\begin{center}Table 5. The number of primes $p\leq 10^7$ such that there is exactly $i$ occurrences of the digit 1 in the continued fraction expansion of $\sqrt{p}$.\end{center}

Based on our numerical calculations we formulate the following
\begin{ques}
Is it true that $\cal{L}_{1}=\{3\}, \cal{L}_{3}=\{7\}$ and $\cal{L}_{2k+1}=\emptyset$ for $k\geq 2$?
\end{ques}

\noindent {\bf Acknowledgments}
We are grateful for remarks of prof. Schinzel concerning the initial form of the manuscript. We thank also the anonymous referee for careful reading of the manuscript and useful remarks improving the presentation .

\vskip 1cm

\noindent Piotr Miska, Jagiellonian University, Faculty of Mathematics and Computer Science, Institute of
Mathematics, {\L}ojasiewicza 6, 30-348 Krak\'ow, Poland; email: piotrmiska91@gmail.com

\noindent Maciej Ulas, Jagiellonian University, Faculty of Mathematics and Computer Science, Institute of
Mathematics, {\L}ojasiewicza 6, 30-348 Krak\'ow, Poland; email: maciej.ulas@uj.edu.pl

\end{document}